\documentclass{amsart}
\usepackage{amsfonts}
\usepackage{amssymb}
\usepackage{amscd}

\newtheorem{theorem}{Theorem}[section]
\newtheorem{corollary}[theorem]{Corollary}

\newtheorem{lemma}[theorem]{Lemma}
\newtheorem{proposition}[theorem]{Proposition}

\newtheorem{remark}{Remark}[section]

\numberwithin{equation}{section}

\input{tcilatex}
\input tcilatex

\begin{document}
\title[Lifespan of solutions for NLS]{Lifespan of solutions for the
nonlinear Schr\"{o}dinger equation without gauge invariance}
\author[M. Ikeda ]{Masahiro Ikeda}

\begin{abstract}
We study the lifespan of solutions for the nonlinear Schr\"{o}dinger
equation 
\begin{equation}
i\partial _{t}u+\Delta u=\lambda |u|^{p},\quad \left( t,x\right) \in \left[
0,T\right) \times \mathbb{R}^{n},  \tag{NLS}  \label{0.1}
\end{equation}%
with the initial condition, where $1<p\leq 1+2/n$ and $\lambda \in \mathbb{C}%
\setminus \{0\}.$ Our main aim in this paper is to prove an upper bound of
the lifespan in the subcritical case $1<p<1+2/n.$
\end{abstract}

\maketitle

\section{\label{S1} Introduction}

In this paper, we study the initial value problem for the nonlinear Schr\"{o}%
dinger equation (NLS) with a non-gauge invariant power nonlinearity: 
\begin{equation}
i\partial _{t}u+\Delta u=\lambda \left\vert u\right\vert ^{p},\quad \left(
t,x\right) \in \left[ 0,T\right) \mathbf{\times }\mathbb{R}^{n},
\label{eq11}
\end{equation}%
with the initial condition 
\begin{equation}
u\left( 0,x\right) =\varepsilon f\left( x\right) ,\text{ \ }x\in \mathbb{R}%
^{n},  \label{eq12}
\end{equation}%
where $T>0,$ $1<p\leq 1+2/n,$ $u$ is a complex-valued unknown function of $%
\left( t,x\right) ,$ $\lambda \in \mathbb{C}\backslash \left\{ 0\right\} ,$ $%
f$ is a given complex-valued function, $\varepsilon >0$ is a small parameter.

It is well known that local well-posedness holds for (\ref{eq11})-(\ref{eq12}%
) in several Sobolev spaces $H^{s}$ $\left( s\geq 0\right) $ (see e.g. \cite%
{Caze03, Tsu87} and the references therein). However, there had been no
results about global existence of solutions for (\ref{eq11})-(\ref{eq12}) in
the case of $1<p\leq 1+2/n.$ It is also well known that when $p\geq p_{s},$
where $p_{s}$ is the Strauss exponent (see \cite{Strauss81}),
\textquotedblleft small data global existence result\textquotedblright\
holds (see also \cite{Caze03}). Recently, in paper \cite{MIYWpre}, blow-up
solutions for (\ref{eq11})-(\ref{eq12}) were constructed in the case of $%
1<p\leq 1+2/n$ under a suitable initial data. To construct blow-up
solutions, they have to choose the shape of the initial data, though the
size of the data may be small. But since they used a contradiction argument
to construct a blow-up solution, the mechanism of the blow-up solution (e.g.
estimate of the lifespan, blow-up speed etc.) has not been known. Motivated
by their result, we decided to consider the lifespan of the local solution
for (\ref{eq11})-(\ref{eq12}). Especially, our main aim of the present paper
is to give an upper bound of the lifespan in the subcritical case $1<p<1+2/n.
$ We note that the optimality of the lifespan is still open. And, in the
critical case $p=1+2/n,$ an upper bound of the lifespan is also still not
known. We also remark that it is open what happens in the case of $%
1+2/n<p\leq p_{s}.$ (For more recent information of blow-up results of NLS,
see e.g. \cite{Oh12}, \cite{OzaSnaPre}, \cite{PRa} and the references
therein.)

\section{\label{S2} Known Results and Main Result}

First, we recall the local existence result for the integral equation in $%
L^{2}$-framework:%
\begin{equation}
u\left( t\right) =\varepsilon U\left( t\right) f-i\lambda
\int_{0}^{t}U\left( t-s\right) \left\vert u\right\vert ^{p}ds,  \label{eq21}
\end{equation}%
which is associated with (\ref{eq11})-(\ref{eq12}), where $U(t)=\exp
(it\Delta )$ is the free evolution group to the Schr\"{o}dinger equation.

\begin{proposition}[Tsutsumi \protect\cite{Tsu87}]
\label{Prop21} Let $1<p<1+4/n,\lambda \in \mathbb{C},$ $\varepsilon \geq 0$
and $f\in L^{2}.$ Then there exist a positive time $T=T\left( \left\Vert
f\right\Vert _{L^{2}},\varepsilon \right) >0$ and a unique solution $u\in
C\left( \left[ 0,T\right) ;L^{2}\right) \cap L_{t}^{r}\left( 0,T;L_{x}^{\rho
}\right) $ of \textrm{(\ref{eq21})}, where $r,\rho $ are defined by $\rho
=p+1$ and $2/r=n/2-n/\rho .$
\end{proposition}

The above solution $u$ is called \textquotedblleft $L^{2}$%
-solution\textquotedblright . Our next concern is the estimate of the
lifespan. Let $T_{\varepsilon }$ be the maximal existence time (lifespan) of
the solution, that is,%
\begin{eqnarray*}
T_{\varepsilon } &\equiv &\sup \left\{ T\in \left( 0,\infty \right] ;\text{
there exists a unique solution }u\text{ to (\ref{eq21})}\right. \\
&&\left. \text{such that }u\in C\left( \left[ 0,T\right) ;L^{2}\right) \cap
L_{t}^{r}\left( 0,T;L_{x}^{\rho }\right) \right\} ,
\end{eqnarray*}%
where $r,\rho $ are as in Proposition \ref{Prop21}. The lower bound of the
lifespan follows from the proposition immediately.

\begin{corollary}
\bigskip \label{Cor 2.1} Under the same assumptions as in Proposition \ref%
{Prop21}, the estimate is valid%
\begin{equation*}
T_{\varepsilon }\geq C\varepsilon ^{1/\omega },
\end{equation*}%
where $\omega =n/4-1/\left( p-1\right) $ and $C=$ $C\left( n,p,\left\Vert
f\right\Vert _{L^{2}}\right) $ is a positive constant$.$
\end{corollary}

The next interest is an upper bound of the lifespan. In \cite{MIYWpre}, it
was proved that $T_{\varepsilon }$ must be finite for suitable initial data.
To recall their result, we introduce some notations: $\lambda _{1}=\func{Re}%
\lambda ,$ $\lambda _{2}=\func{Im}\lambda ,$ $f_{1}=\func{Re}f$ and $f_{2}=%
\func{Im}f.$

We impose the additional assumptions on the data:%
\begin{equation}
\text{\textquotedblleft }f_{1}\in L^{1},\ \lambda _{2}\int_{\mathbb{R}%
^{n}}f_{1}(x)dx>0\text{\textquotedblright }\ \text{or}\ \text{%
\textquotedblleft }f_{2}\in L^{1},\ \lambda _{1}\int_{\mathbb{R}%
^{n}}f_{2}\left( x\right) dx<0\text{\textquotedblright }.  \label{eq22}
\end{equation}%
Then the following is valid:

\begin{proposition}[Ikeda and Wakasugi \protect\cite{MIYWpre}]
\label{Prop22} Let $1<p\leq 1+2/n,$ $\lambda \in \mathbb{C}\setminus \left\{
0\right\} ,$ $\varepsilon >0$ and $f\in L^{2}.$ If $f$ satisfies (\ref{eq22}%
), then $T_{\varepsilon }<\infty .$ Moreover, the $L^{2}$-norm of the local
solution blows up in finite time;%
\begin{equation}
\lim_{t\rightarrow T_{\varepsilon }-0}\left\Vert u\left( t\right)
\right\Vert _{L^{2}}=\infty .  \label{eq23}
\end{equation}
\end{proposition}

It is well known that the similar result holds for the corresponding
nonlinear heat equation and the damped wave equation. Proposition \ref%
{Prop22} can be said to be the NLS version. We will give a proof of the
proposition different from \cite{MIYWpre} in Appendix.

\begin{remark}
\label{Remark 2.2} In \cite{MIYWpre}, in order to prove $T_{\varepsilon
}<\infty ,$ a contradiction argument based on papers \cite{Zha99}, \cite%
{Zha01} was used. Therefore, an upper bound of the lifespan was not obtained.
\end{remark}

Next, we state our main result in this paper, which gives an upper bound of
the lifespan. We put the more additional assumption on the data:%
\begin{equation}
\left\{ 
\begin{array}{c}
\text{\textquotedblleft }f_{1}\in L^{1},\text{ }\lambda _{2}f_{1}\left(
x\right) \geq \left\vert x\right\vert ^{-k},\text{ }\left\vert x\right\vert
>1\text{\textquotedblright } \\ 
\text{or \textquotedblleft }f_{2}\in L^{1},\text{ }-\lambda _{1}f_{2}\left(
x\right) \geq \left\vert x\right\vert ^{-k},\text{ }\left\vert x\right\vert
>1\text{\textquotedblright }%
\end{array}%
\right. \text{\ }  \label{eq24}
\end{equation}%
where $n<k<2/\left( p-1\right) ,$ which exists, if $1<p<1+2/n.$ We also note
that the function of the right hand side of (\ref{eq24}) belongs to $%
L^{1}\cap L^{2}.$ Then the following is valid;

\begin{theorem}
\label{Theorem 2} Let $1<p<1+2/n,$ $\lambda \in \mathbb{C}\setminus \left\{
0\right\} $ and $f\in L^{2}.$ If $f$ satisfies (\ref{eq24}), then there
exist $\varepsilon _{0}>0$ and positive constant $C=C\left( k,p,\lambda
\right) $ such that%
\begin{equation*}
T_{\varepsilon }\leq C\varepsilon ^{1/\kappa }
\end{equation*}
for any $\varepsilon \in \left( 0,\varepsilon _{0}\right) $ where $\kappa
\equiv k/2-1/\left( p-1\right) .$
\end{theorem}

\begin{remark}
\label{Remark 2.1} We note that there is a gap between the lower bound (see
Corollary \ref{Cor 2.1}) and the upper bound in $L^{2}$-framework, that is $%
\kappa >\omega .$
\end{remark}

Next, we consider the possibility to fill the gap in other frameworks.
Especially, we consider the local existence for (\ref{eq21}) in $H^{1}\cap
L^{1+1/p}$-framework.

Let $\gamma >\frac{2}{n}\frac{p+1}{p-1}.$ The following result is valid:

\begin{proposition}
\label{Prop 2.2} Let $1<p<1+4/\left( n-2\right) ,$ $\lambda \in \mathbb{C},$ 
$\varepsilon \geq 0$ and $f\in H^{1}\cap L^{1+1/p}.$ Then there exist a
positive time $T=T\left( \varepsilon ,\left\Vert f\right\Vert _{H^{1}\cap
L^{1+1/p}}\right) $ and a unique solution $u\in C\left( \left[ 0,T\right)
;H^{1}\right) \cap L_{t}^{\gamma }\left( 0,T;L_{x}^{\rho }\right) \cap
L_{t}^{r}\left( 0,T;W^{1,\rho }\right) $ for (\ref{eq21}), where $\rho =p+1,$
$r$ is given by $2/r=n\left( 1/2-1/\rho \right) .$
\end{proposition}

This proposition can be proved in the almost same mannar as in the proof of
Theorem 6.3.2 in \cite{Caze03} (see also \cite{Strauss81}). From the
proposition, a lower bound of the lifespan also follows immediately. Let $%
\tilde{T}_{\varepsilon }$ be the maximal existence time of the local
solution obtained in Proposition \ref{Prop 2.2}.

\begin{corollary}
\label{Cor 3.2} Under the same assumptions as in Proposition \ref{Prop 2.2},
the inequality is valid: 
\begin{equation*}
\tilde{T}_{\varepsilon }\geq C\varepsilon ^{1/\sigma },
\end{equation*}
where $\sigma =\frac{1}{\gamma }+\frac{n}{2\left( p+1\right) }-\frac{1}{p-1}$
and $C=C\left( n,p,\left\Vert f\right\Vert _{H^{1}\cap L^{1+1/p}}\right) $
is some positive constant.
\end{corollary}

\begin{remark}
\label{Remark 3} The same conclusion as in Theorem \ref{Theorem 2} holds
even if $\left( T_{\varepsilon },L^{2}\right) $ is replaced by $\left( 
\tilde{T}_{\varepsilon },H^{1}\cap L^{1+1/p}\right) .$ In Corollary \ref{Cor
3.2}, since%
\begin{equation*}
\sigma \rightarrow \frac{n}{2}\left( 1-\frac{1}{p+1}\right) -\frac{1}{p-1}
\end{equation*}%
as $r\rightarrow \frac{2}{n}\frac{p+1}{p-1},$ we can see that there is also
gap between the upper bound and the lower bound of the lifespan in $%
H^{1}\cap L^{1+1/p}$-framework, though the lower estimate is improved, i.e. $%
\kappa >\sigma >\omega $ as $r\rightarrow \frac{2}{n}\frac{p+1}{p-1}.$
\end{remark}

\begin{remark}
\label{Remark 2} In the critical case $p=1+2/n,$ we do not know an upper
bound of lifespan for (\ref{eq11})-(\ref{eq12}).
\end{remark}

At the end of this section, we mention the strategy of the proof of Theorem %
\ref{Theorem 2}. We will use a test-function method based on papers \cite%
{HJK03}, \cite{FS10}. In \cite{HJK03}, \cite{FS10}, upper bounds of lifespan
for some parabolic equations were obtained. However, their arguement does
not be applicable to the present NLS directly. Since solutions for NLS are
complex-valued, the constant $\lambda $ in front of the nonlinearity is a
complex number and especially, the appropriate function spaces for NLS
differs from that of those parabolic equations. To overcome these
difficulties, we will consider the real part or imaginary part for the
equation and reconsider the problem under the suitable function spaces $%
L^{2} $ $\left( \text{or }H^{1}\cap L^{1+1/p}\right) $ to NLS, so that we
can use the local existence theorem.

\section{\label{S2.5} Integral inequalities}

In this section, we prepare some integral inequalities. Before doing so, we
introduce the non-negative smooth function $\phi $ as follows, which was
constructed in the papers \cite{ChGa01}, \cite{GaPo02}:%
\begin{equation*}
\phi \left( x\right) =\phi \left( \left\vert x\right\vert \right) ,\text{ }%
\phi \left( 0\right) =1,\text{ }0<\phi \left( x\right) \leq 1\text{ for }%
\left\vert x\right\vert >0,
\end{equation*}%
where $\phi \left( \left\vert x\right\vert \right) $ is decreasing of $%
\left\vert x\right\vert $ and $\phi \left( \left\vert x\right\vert \right)
\rightarrow 0$ as $\left\vert x\right\vert \rightarrow \infty $ sufficiently
fast. Moreover, there exists $\mu >0$ such that%
\begin{equation}
\left\vert \Delta \phi \right\vert \leq \mu \phi ,\text{ \ }x\in \mathbb{R}%
^{n},  \label{2.1}
\end{equation}%
and $\left\Vert \phi \right\Vert _{L^{1}}=1.$ This can be done by letting $%
\phi \left( r\right) =e^{-r^{\nu }}$ for $r\gg 1$ with $\nu \in \left( 0,1%
\right] $ and extending $\phi $ to $\left[ 0,\infty \right) $ by a smooth
approximation. Let $\theta $ be suffuciently large and%
\begin{equation*}
\eta \left( t\right) =\eta _{S,T}\left( t\right) =\left\{ 
\begin{array}{l}
0, \\ 
\left( 1-\left( t-S\right) /\left( T-S\right) \right) ^{\theta }, \\ 
1,%
\end{array}%
\begin{array}{l}
\text{if }t>T, \\ 
\text{if }S\leq t\leq T, \\ 
\text{if }t<S,%
\end{array}%
\right.
\end{equation*}%
where $0\leq S<T.$ Furthermore, set $\eta _{R}\left( t\right) =\eta \left(
t/R^{2}\right) ,$ $\phi _{R}\left( x\right) =\phi \left( x/R\right) $ and $%
\psi _{R}\left( t,x\right) =\eta _{R}\left( t\right) \phi _{R}\left(
x\right) $ for $R>0.$

First, we reduce the integral equation (\ref{eq21}) into the weak form.

\begin{lemma}
\label{Lemma 3.1} Let $u$ be an $L^{2}$-solution of \textrm{(\ref{eq11})-(%
\ref{eq12})} on $[0,T).$ Then $u$ satisfies%
\begin{eqnarray}
\lefteqn{\int_{\left[ 0,T\right) \times \mathbb{R}^{n}}u(-i\partial
_{t}\left( \psi _{R}\right) +\Delta \left( \psi _{R}\right) )dxdt}  \notag \\
&=&i\varepsilon \int_{\mathbb{R}^{n}}f\left( x\right) \psi _{R}\left(
0,x\right) dx+\lambda \int_{\left[ 0,T\right) \times \mathbb{R}%
^{n}}\left\vert u\right\vert ^{p}\psi _{R}dxdt.  \label{3.1}
\end{eqnarray}
\end{lemma}

This lemma can be proved in the same manner as the proof of Proposition 3.1
in \cite{MIYWpre}.

Next, we will lead a integral inequality. Hereafter we only consider the
case of $\lambda _{1}>0$ for simplicity. The other cases can be treated in
the almost same way (see Remark \ref{Remark 3.1}).

We introduce some functions:%
\begin{eqnarray*}
I_{R}\left( S,T\right) &=&\int_{\left[ SR^{2},TR^{2}\right) \times \mathbb{R}%
^{n}}\left\vert u\right\vert ^{p}\psi _{R}dxdt, \\
J_{R} &=&\varepsilon \int_{\mathbb{R}^{n}}-f_{2}\left( x\right) \phi \left(
x/R\right) dx
\end{eqnarray*}%
and%
\begin{eqnarray*}
A\left( S,T\right) &=&\left( \int_{\left[ S,T\right) \times \mathbb{R}%
^{n}}\left\vert \partial _{t}\eta \left( t\right) \right\vert ^{q}\eta
\left( t\right) ^{-1/\left( p-1\right) }\phi \left( x\right) dxdt\right)
^{1/q}, \\
B\left( S,T\right) &=&\left( \int_{\left[ 0,T\right) \times \mathbb{R}%
^{n}}\eta _{S,T}\left( t\right) \phi \left( x\right) dxdt\right) ^{1/q},
\end{eqnarray*}%
where $q=p/\left( p-1\right) .$ By the direct computation, we have%
\begin{equation}
A\left( S,T\right) =\theta \left\{ \theta -1/\left( p-1\right) \right\}
^{-1/q}\left( T-S\right) ^{-1/p},\text{ }B\left( S,T\right) =\left( S+\frac{%
T-S}{\theta +1}\right) ^{1/q}.  \label{3.2a}
\end{equation}

We have the following:

\begin{lemma}
\label{Lemma3.2} Let $u$ be an $L^{2}$-solution of \textrm{(\ref{eq11})-(\ref%
{eq12})} on $\left[ 0,T_{\ast }\right) .$ Then the inequality holds%
\begin{equation}
\lambda _{1}I_{R}\left( 0,T\right) +J_{R}\leq R^{s}\left\{ I_{R}\left(
S,T\right) ^{1/p}A\left( S,T\right) +\mu I_{R}\left( 0,T\right)
^{1/p}B\left( S,T\right) \right\}  \label{3.4}
\end{equation}%
for any $0\leq S<T$ and $R>0$ with $TR^{2}<T_{\ast },$ where $s=-2+\left(
2+n\right) /q.$
\end{lemma}

\begin{proof}
Since $u$ is $L^{2}$-solution on $\left[ 0,T_{\ast }\right) $ and $%
TR^{2}<T_{\ast },$ by Lemma \ref{3.1}, we have%
\begin{eqnarray}
\lefteqn{\lambda \int_{\left[ 0,TR^{2}\right) \times \mathbb{R}%
^{n}}\left\vert u\right\vert ^{p}\psi _{R}dxdt+i\varepsilon \int_{\mathbb{R}%
^{n}}f\left( x\right) \psi _{R}\left( 0,x\right) dx}  \notag \\
&=&\int_{\left[ 0,TR^{2}\right) \times \mathbb{R}^{n}}u(-i\partial
_{t}\left( \psi _{R}\right) +\Delta \left( \psi _{R}\right) )dxdt.
\label{3.5}
\end{eqnarray}%
Note that $\lambda _{1}>0,$ by taking real part as the above identity, we
obtain%
\begin{eqnarray}
\lambda _{1}I_{R}\left( 0,T\right) +J_{R} &=&\int_{\left[ 0,TR^{2}\right)
\times \mathbb{R}^{n}}\func{Re}u\left( -i\partial _{t}\left( \psi
_{R}\right) +\Delta \left( \psi _{R}\right) \right) dxdt  \notag \\
&\leq &\int_{\left[ 0,TR^{2}\right) \times \mathbb{R}^{n}}\left\vert
u\right\vert \left\{ \left\vert \partial _{t}\left( \psi _{R}\right)
\right\vert +\left\vert \Delta \left( \psi _{R}\right) \right\vert \right\}
dxdt  \notag \\
&\equiv &K_{R}^{1}+K_{R}^{2}  \label{3.6}
\end{eqnarray}%
We note that $\left( \partial _{t}\eta \right) \left( t\right) =0$ except on 
$\left( S,T\right) .$ By using the identity 
\begin{equation*}
\partial _{t}\psi _{R}\left( t,x\right) =R^{-2}\phi _{R}\left( x\right)
\left( \partial _{t}\eta \right) \left( t/R^{2}\right)
\end{equation*}%
and the H\"{o}lder inequality, we can get%
\begin{eqnarray}
K_{R}^{1} &=&R^{-2}\int_{\left[ SR^{2},TR^{2}\right) \times \mathbb{R}%
^{n}}\left\vert u\right\vert \eta _{R}^{1/p}\left\vert \left( \partial
_{t}\eta \right) \left( t/R^{2}\right) \right\vert \eta _{R}^{-1/p}\phi
_{R}dxdt  \notag \\
&\leq &R^{-2}I_{R}\left( S,T\right) ^{1/p}\left( \int_{\left[
SR^{2},TR^{2}\right) \times \mathbb{R}^{n}}\left\vert \left( \partial
_{t}\eta \right) \left( t/R^{2}\right) \right\vert ^{q}\eta _{R}^{-1/\left(
p-1\right) }\phi _{R}dxdt\right) ^{1/q}  \notag \\
&=&I_{R}\left( S,T\right) ^{1/p}A\left( S,T\right) R^{s},  \label{3.7}
\end{eqnarray}%
where we have used the changing variables with $t/R^{2}=t^{\prime }$ and $%
x/R=x^{\prime }$ to obtain the last identity. Next, by the identity $\Delta
\left( \phi \left( x/R\right) \right) =R^{-2}\left( \Delta \phi \right)
\left( x/R\right) ,$ the H\"{o}lder inequality and the estimate (\ref{2.1}),
we have%
\begin{eqnarray}
K_{R}^{2} &=&R^{-2}\int_{\left[ 0,TR^{2}\right) \times \mathbb{R}%
^{n}}\left\vert u\right\vert \eta \left( t/R^{2}\right) \left\vert \left(
\Delta \phi \right) \left( x/R\right) \right\vert dxdt  \notag \\
&\leq &\mu R^{-2}\int_{\left[ 0,TR^{2}\right) \times \mathbb{R}%
^{n}}\left\vert u\right\vert \psi _{R}dxdt  \notag \\
&\leq &\mu R^{-2}I_{R}\left( 0,T\right) ^{1/p}\left( \int_{\left[
0,TR^{2}\right) \times \mathbb{R}^{n}}\psi _{R}dxdt\right) ^{1/q}  \notag \\
&=&\mu I_{R}\left( 0,T\right) ^{1/p}B\left( S,T\right) R^{s},  \label{3.8}
\end{eqnarray}%
where we have used the changing variables again. By combining the estimates (%
\ref{3.6})-(\ref{3.8}), we have the conclusion.
\end{proof}

\begin{remark}
\label{Remark 3.1} We remark the other cases different from $\lambda _{1}>0.$
For example, when $\lambda _{2}>0,$ by taking the imaginary part as (\ref%
{3.5}), an estimate similar to (\ref{3.2}) can be obtained.
\end{remark}

Next, we give the upper bound of $J_{R}.$ Let $\sigma >0$ and $0<\omega <1.$
We introduce the function%
\begin{equation}
\Psi \left( \sigma ,\omega \right) \equiv \max_{x\geq 0}\left( \sigma
x^{\omega }-x\right) =\left( 1-\omega \right) \omega ^{\frac{\omega }{%
1-\omega }}\sigma ^{\frac{1}{1-\omega }}.  \label{2.9a}
\end{equation}%
We also denote $I_{R}\left( T\right) =I_{R}\left( 0,T\right) ,$ $A\left(
T\right) =A\left( 0,T\right) ,$ $B\left( T\right) =B\left( 0,T\right) $ and%
\begin{equation*}
D\left( T\right) =A\left( T\right) +\mu B\left( T\right) ,
\end{equation*}%
for simplicity. The following estimates are valid:

\begin{lemma}
\label{Lemma 3.3} Let $u$ be an $L^{2}$-solution of (\ref{eq11})-(\ref{eq12}%
) on $\left[ 0,T_{\ast }\right) .$ Then the estimate%
\begin{equation}
J_{R}\leq \lambda _{1}\Psi \left( D\left( T\right) R^{s}/\lambda
_{1},1/p\right)  \label{2.10}
\end{equation}%
holds for any $T>0,R>0$ with $TR^{2}<T_{\ast },$ where $s=-2+\left(
2+n\right) /q.$ Moreover, if $T_{\ast }=\infty ,$ that is $u$ is a global
solution, then the inequality is valid:%
\begin{equation}
\limsup_{R\rightarrow \infty }R^{-sq}J_{R}\leq \left( \mu /\lambda
_{1}\right) ^{1/\left( p-1\right) }.  \label{2.11}
\end{equation}
\end{lemma}

The proof of this lemma was based on that of Theorem 3.3 in \cite{HJK03} and
Theorem 2.2 in \cite{FS10}.

\begin{proof}
Since $u$ is an $L^{2}$-solution on $\left[ 0,T_{\ast }\right) ,$ by using (%
\ref{3.4}) with $S=0,$ we obtain%
\begin{equation*}
J_{R}\leq R^{s}D\left( T\right) I_{R}\left( T\right) ^{1/p}-\lambda
_{1}I_{R}\left( T\right) \leq \lambda _{1}\Psi \left( D\left( T\right)
R^{s}/\lambda _{1},1/p\right) ,
\end{equation*}%
which is exactly (\ref{2.10}).

Next, we will prove (\ref{2.11}) under the assumption $T_{\ast }=\infty .$
By (\ref{2.9a}) and (\ref{2.10}), we have%
\begin{eqnarray}
J_{R} &\leq &\lambda _{1}\Psi \left( D\left( T\right) R^{s}/\lambda
_{1},1/p\right)  \notag \\
&=&\lambda _{1}\left( 1-1/p\right) \left( 1/p\right) ^{\frac{1/p}{1-1/p}%
}\left\{ D\left( T\right) R^{s}/\lambda _{1}\right\} ^{\frac{1}{1-1/p}} 
\notag \\
&=&C_{1}R^{sq}D\left( T\right) ^{q},  \label{2.12}
\end{eqnarray}%
for any $T>0,R>0,$ where $C_{1}=\lambda _{1}^{-1/\left( p-1\right) }\left(
p-1\right) \left( 1/p\right) ^{q}.$ This inequality implies%
\begin{equation}
\limsup_{R\rightarrow \infty }R^{-sq}J_{R}\leq C_{1}\left\{
\inf_{T>0}D\left( T\right) \right\} ^{q}.  \label{2.13}
\end{equation}%
Next, we will estimate $D\left( T\right) .$ Set%
\begin{equation*}
a_{p}=\frac{\theta }{\left\{ \theta -1/\left( p-1\right) \right\} ^{1/q}},%
\text{ \ }b_{p}=\frac{\mu }{\left( \theta +1\right) ^{1/q}}.
\end{equation*}%
Remembering the identities (\ref{3.2a}), we can rewrite $D\left( T\right) $
as%
\begin{equation}
D\left( T\right) =a_{p}T^{-1/p}+b_{p}T^{1/q}.  \label{2.13b}
\end{equation}%
Since%
\begin{eqnarray}
\min_{T>0}D\left( T\right) &=&p\left( p-1\right)
^{-1/q}a_{p}^{1/q}b_{p}^{1/p}  \notag \\
&=&\frac{\mu ^{1/p}p\left( p-1\right) ^{-1/q}\theta ^{1/q}}{\left\{ \theta
-1/\left( p-1\right) \right\} ^{1/q^{2}}\left( 1+\theta \right) ^{1/\left(
pq\right) }},  \label{2.13a}
\end{eqnarray}%
we have%
\begin{equation}
\lim_{\theta \rightarrow \infty }\min_{T>0}D\left( T\right) =\mu
^{1/p}p\left( p-1\right) ^{-1/q}.  \label{2.14}
\end{equation}%
Finally, by combining (\ref{2.13})-(\ref{2.14}), we obtain (\ref{2.11}),
which completes the proof of the lemma.
\end{proof}

\section{\protect\bigskip \label{S3} Upper bound of lifespan}

In this section, we give a proof of Theorem \ref{Theorem 2}, which implies
an upper bound of the lifespan for the local $L^{2}$-solution. We also
consider the case of $\lambda _{1}>0$ only. The other cases can be treated
in the almost same manner. When $\lambda _{1}>0,$ we may assume that $f_{2}$
satisfies

\begin{equation}
\text{\textquotedblleft }f_{1}\in L^{1},\text{ }\lambda _{2}f_{1}\left(
x\right) \geq \left\vert x\right\vert ^{-k},\text{ }\left\vert x\right\vert
>1\text{\textquotedblright\ or \textquotedblleft }f_{2}\in L^{1},\text{ }%
-\lambda _{1}f_{2}\left( x\right) \geq \left\vert x\right\vert ^{-k},\text{ }%
\left\vert x\right\vert >1\text{\textquotedblright }  \label{3.1a}
\end{equation}%
where $n<k<2/\left( p-1\right) .$

\begin{proof}
First, we note that by Corollary \ref{Cor 2.1}, there exists $\varepsilon
_{0}>0$ such that $T_{\varepsilon }>1$ for any $\varepsilon \in \left(
0,\varepsilon _{0}\right) .$ Moreover, since $1<p<1+2/n$ and $f$ satisfies (%
\ref{eq24}), by Proposition \ref{Prop22}, we also find $T_{\varepsilon
}<\infty .$

Next, we consider the lower bound of $J_{R}.$ By changing variables and (\ref%
{3.1a}), we have%
\begin{eqnarray*}
J_{R} &=&\varepsilon R^{n}\int_{\mathbb{R}^{n}}-f_{2}\left( Rx\right) \phi
\left( x\right) dx \\
&\geq &\varepsilon R^{n}\int_{\left\vert x\right\vert \geq 1/R}-f_{2}\left(
Rx\right) \phi \left( x\right) dx \\
&\geq &\varepsilon R^{n-k}\int_{\left\vert x\right\vert \geq 1/R}\left\vert
x\right\vert ^{-k}\phi \left( x\right) dx \\
&\geq &\varepsilon R^{n-k}\int_{\left\vert x\right\vert \geq
1/R_{0}}\left\vert x\right\vert ^{-k}\phi \left( x\right)
dx=C_{k}\varepsilon R^{n-k}.
\end{eqnarray*}%
for any $R>R_{0}>0,$ where $R_{0}$ is a constant independent of $%
R,\varepsilon $ and defined later.

Next, let $\tau \in \left( 1,T_{\varepsilon }\right) $ and $R>R_{0}.$ By
using (\ref{2.12}) with $T=\tau R^{-2},$ we have%
\begin{equation}
\varepsilon \leq C_{k}^{-1}C_{1}\left\{ R^{s}D\left( \tau R^{-2}\right)
\right\} ^{q}R^{-n+k}\equiv C_{2}H\left( \tau ,R\right) ,  \label{3.2}
\end{equation}%
where $C_{2}=C_{k}^{-1}C_{1}.$ By (\ref{2.13b}), we can rewrite $H$ as%
\begin{equation}
H\left( \tau ,R\right) =R^{-n+k}\left\{ D\left( \tau R^{-2}\right)
R^{s}\right\} ^{q}=\left\{ a_{p}\tau ^{-1/p}R^{\alpha _{1}}+b_{p}\tau
^{1/q}R^{-\alpha _{2}}\right\} ^{q},  \label{3.3a}
\end{equation}%
where $\alpha _{1}=k/q,$ $\alpha _{2}=2-k/q.$

Now we derive some properties on $H\left( \tau ,R\right) .$ We assume that
we can find a function $G\left( \tau \right) $ satisfying the following two
properties: The first one is that for any $\tau \in \left( 1,T_{\varepsilon
}\right) $ and any $R>R_{0},$ $H\left( \tau ,R\right) \geq G\left( \tau
\right) $ and the other one is that for any $\tau \in \left(
1,T_{\varepsilon }\right) ,$ there exists $R_{\tau }>R_{0}$ such that $%
H\left( \tau ,R_{\tau }\right) =G\left( \tau \right) .$ Then (\ref{3.2})
holds for any $\tau \in \left( 1,T_{\varepsilon }\right) ,$ $R>R_{0}$ if and
only if%
\begin{equation}
\varepsilon \leq C_{2}G\left( \tau \right) ,  \label{3.3}
\end{equation}%
for any $\tau \in \left( 1,T_{\varepsilon }\right) .$ Actually, we can find
such function $G\left( \tau \right) $ as follows. Set $y=R^{\alpha
_{1}+\alpha _{2}}=R^{2},$\ $\beta _{1}=\alpha _{2}/\left( \alpha _{1}+\alpha
_{2}\right) =\alpha _{2}/2$ and $h\left( \tau ,y\right) =a_{p}y^{1-\beta
_{1}}+b_{p}y^{-\beta _{1}}\tau .$ Then we can rewrite%
\begin{equation*}
H\left( \tau ,R\right) =\tau ^{-1/\left( p-1\right) }h\left( \tau ,y\right)
^{q}.
\end{equation*}%
Denote $\sigma =\sigma \left( y\right) =a_{p}b_{p}^{-1}\left( 1-\beta
_{1}\right) \beta _{1}^{-1}y,$ $g\left( \tau \right) =\left\{
a_{p}y^{1-\beta _{1}}\sigma ^{\beta _{1}-1}+b_{p}y^{-\beta _{1}}\sigma
^{\beta _{1}}\right\} \tau ^{1-\beta _{1}}$ and%
\begin{equation*}
G\left( \tau \right) =\tau ^{-1/\left( p-1\right) }g\left( \tau \right) ^{q}.
\end{equation*}%
It is easy to check $0<\beta _{1}<1.$ Then $\zeta =g\left( \tau \right) $ is
a convex. Furthermore, $\zeta =h\left( \tau ,y\right) $ is a tangent line of 
$\zeta =g\left( \tau \right) $ at the point of $\left( \sigma ,g\left(
\sigma \right) \right) .$ Therefore, we can see that $h\left( \tau ,y\right)
\geq g\left( \tau \right) ,$ for all $\tau >0.$ Hence $H\left( \tau
,R\right) \geq G\left( \tau \right) ,$ for any $\tau ,R>0.$ Here, we choose $%
R_{0}$ as $0<R_{0}<\left\{ \sigma ^{-1}\left( 1\right) \right\} ^{1/2}.$
Then for any $\tau \in \left( 1,T_{\varepsilon }\right) ,$ if we set $%
R_{\tau }=\left\{ \sigma ^{-1}\left( \tau \right) \right\} ^{1/2},$ that is,%
\begin{equation*}
R_{\tau }=\left\{ a_{p}^{-1}b_{p}\beta _{1}\left( 1-\beta _{1}\right)
^{-1}\tau \right\} ^{1/2}\left( >R_{0}\right) ,
\end{equation*}%
we have $H\left( \tau ,R_{\tau }\right) =G\left( \tau \right) .$ On the
other hand, by the direct computation, we have%
\begin{equation}
G\left( \tau \right) =C_{3}\tau ^{\kappa },  \label{3.4a}
\end{equation}%
where $\kappa =k/2-1/\left( p-1\right) $ and $C_{3}=C_{3}\left( \theta
,p\right) >0$ is constant dependent only on $\theta ,p.$ By combining (\ref%
{3.3}) and (\ref{3.4a}), we have $\varepsilon \leq C_{4}\tau ^{\kappa },$
with $C_{4}=C_{2}C_{3}>0.$ From the assumption $k<2/\left( p-1\right) ,$ we
obtain $\kappa <0.$ Therefore, by (\ref{3.4a}), we can get%
\begin{equation*}
\tau \leq C\varepsilon ^{1/\kappa }
\end{equation*}%
for any $\tau \in \left( 1,T_{\varepsilon }\right) ,$ with some $C>0.$
Finally, we can get $T_{\varepsilon }\leq C\varepsilon ^{1/\kappa },$ which
completes the proof of the theorem.
\end{proof}

\section{\label{S4} Appendix}

In this Appendix, we give a proof of Proposition \ref{Prop22}, which was
already proved in \cite{MIYWpre}, though the following arguement is
different. We consider the case of $\lambda _{1}>0$ only. In this case, we
may assume that%
\begin{equation*}
f_{2}\in L^{1},\ \int_{\mathbb{R}^{n}}f_{2}\left( x\right) dx<0.
\end{equation*}

\begin{proof}
We use a contradiction argument. We assume that $T_{\varepsilon }=\infty .$
Then we note that there exists a unique global $L^{2}$-solution $u$ for (\ref%
{eq11})-(\ref{eq12}). By the assumption on $f_{2},$ we can see that $J_{R}$
is positive for sufficiently large $R>0.$ In fact, due to $f_{2}\in L^{1},$
by Lebesgue's convergence theorem, we have%
\begin{equation*}
\lim_{R\rightarrow \infty }J_{R}=\varepsilon \int_{\mathbb{R}%
^{n}}-f_{2}\left( x\right) dx>0.
\end{equation*}%
Thus by (\ref{3.4}), we obtain 
\begin{equation*}
I_{R}\left( 0,T\right) \leq CI_{R}\left( 0,T\right) ^{1/p}R^{s},
\end{equation*}%
for any sufficiently large $R,$ with some positive constant $C$ independent
of $R,$ which implies that%
\begin{equation}
I_{R}\left( 0,T\right) \leq CR^{qs}\leq C,  \label{3.9}
\end{equation}%
for any large $R,$ due to $qs\leq 0,$ $\left( \text{i.e. }1<p\leq
1+2/n\right) .$ Therefore, by monotone convergence theorem and letting $%
R\rightarrow \infty $ in (\ref{3.9}), we have 
\begin{equation*}
\,\int_{\left[ 0,\infty \right) \times \mathbb{R}^{n}}\left\vert
u\right\vert ^{p}dxdt<\infty .
\end{equation*}%
In the case of $s<0,$ that is $1<p<1+2/n,$ letting $R$ tend to infinity in (%
\ref{3.9}), we obtain%
\begin{equation*}
\int_{\left[ 0,\infty \right) \times \mathbb{R}^{n}}\left\vert u\right\vert
^{p}dxdt=0.
\end{equation*}%
Hence $u=0$ for a.e $\left( t,x\right) \in \left[ 0,\infty \right) \times 
\mathbb{R}^{n}.$ Finally, letting $R\rightarrow \infty $ in (\ref{3.4}), we
get $\int_{\mathbb{R}^{n}}-f_{2}\left( x\right) dx\leq 0,$ which contradics
to the assumption on $f.$

Next, we consider the critical case $s=0,$ i.e. $p=1+2/n.$ Remembering (\ref%
{3.2a}), if we choose small $S$ and large $\theta $ with $T-S$ bounded, we
obtain%
\begin{equation}
B\left( S,T\right) \leq \left( \int_{\mathbb{R}^{n}}-f_{2}\left( x\right)
dx\right) /\left\{ 2\mu \left( \int_{\left[ 0,\infty \right) \times \mathbb{R%
}^{n}}\left\vert u\right\vert ^{p}dxdt\right) ^{1/p}\right\} .  \label{2.8}
\end{equation}%
By the uniform boundedness (\ref{3.9}) of $I_{R},$ keeping $T-S$ bounded, we
have%
\begin{equation}
\lim_{R\rightarrow \infty }I_{R}\left( S,T\right) ^{1/p}A\left( S,T\right)
=0.  \label{2.9}
\end{equation}%
Finally, letting $R\rightarrow \infty $ in (\ref{3.4}), we obtain%
\begin{equation*}
\lambda _{1}\int_{\left[ 0,\infty \right) \times \mathbb{R}^{n}}\left\vert
u\right\vert ^{p}dxdt+\frac{1}{2}\int_{\mathbb{R}^{n}}-f_{2}\left( x\right)
dx\leq 0,
\end{equation*}%
which also contradicts to the assumption on $\lambda _{1},f_{2}.$ This
completes the proof of the proposition.
\end{proof}

\emph{Acknowledgments}. The author would like to express deep gratitude to
Professor Yoshio Tsutsumi for their many useful sugesstions, comments and
constant encouragement. The author would also like to thank Mr. Kiyotaka
Suzaki and Mr. Yuta Wakasugi for reading our paper carefully and pointing
out some mistakes.

\end{document}